\documentclass[11pt,letterpaper,twoside,reqno]{amsart}

\usepackage{amsmath}
\usepackage{amssymb}
\usepackage{bm}
\usepackage{mathrsfs}
\usepackage[dvips]{graphicx}

\usepackage{color}

\usepackage{url}

\usepackage{supertabular}

\usepackage{alg}

\makeatletter
\def\algspacing{\alg@unmargin}
\makeatother

\newtheorem{thm}{Theorem}

\newtheorem{lemma}[thm]{Lemma}
\newtheorem{prop}[thm]{Proposition}

\theoremstyle{remark}
\newtheorem{rem}[thm]{Remark}

\numberwithin{equation}{section}

\numberwithin{thm}{section}

\DeclareMathAlphabet{\mathsfsl}{OT1}{cmss}{m}{sl}

\newcommand{\term}{\emph}

\newcommand{\cnst}[1]{\mathrm{#1}}

	{\makebox{\phantom{}} \\ %
		\textsc{Input:}\begin{itemize}}%
	{\end{itemize}}

	{\textsc{Output:}\begin{itemize}}%
	{\end{itemize}}

	{\textsc{Procedure:}\begin{enumerate}}%
	{\end{enumerate}}

\renewcommand{\phi}{\varphi}

\newcommand{\eps}{\varepsilon}

\newcommand{\econst}{\mathrm{e}}

\newcommand{\onevct}{\mathbf{e}}

\newcommand{\Id}{\mathbf{I}}

\newcommand{\coll}[1]{\mathscr{#1}}

\newcommand{\Rspace}[1]{\mathbb{R}^{#1}}

\newcommand{\abs}[1]{\left\vert {#1} \right\vert}

\newcommand{\Prob}[1]{\mathbb{P}\left\{ {#1} \right\}}
\newcommand{\Expect}{\operatorname{\mathbb{E}}}

\newcommand{\vct}[1]{\bm{#1}}
\newcommand{\mtx}[1]{\bm{#1}}

\newcommand{\adj}{*}

\newcommand{\diag}{\operatorname{diag}}
\newcommand{\trace}{\operatorname{tr}}

\newcommand{\psdle}{\preccurlyeq}

\newcommand{\norm}[1]{\left\Vert {#1} \right\Vert}
\newcommand{\normsq}[1]{\norm{#1}^2}

\newcommand{\smnorm}[2]{{\bigl\Vert {#2} \bigr\Vert}_{#1}}

\newcommand{\fnorm}[1]{\norm{#1}_{\mathrm{F}}}

\newcommand{\bigO}{\mathrm{O}}

\usepackage{subfigure}
\usepackage{wasysym}

\begin{document}
\title[Analysis of the SRHT]
{Improved Analysis of the \\ Subsampled Randomized Hadamard Transform}

\author{Joel~A.~Tropp}



\thanks{JAT is with Computing \& Mathematical Sciences, MC 305-16, California Inst.~Technology, Pasadena, CA 91125.
E-mail: \url{jtropp@cms.caltech.edu}.
Research supported by ONR award N00014-08-1-0883, DARPA award N66001-08-1-2065, and AFOSR award FA9550-09-1-0643.}

\thanks{This note will appear in \textsl{Advances in Adaptive Data Analysis}, special issue ``Sparse Representation of Multiscale Data and Images.''}

\date{17 June 2010.  Revised 5 November 2010 and 17 July 2011.}

\begin{abstract}
This paper presents an improved analysis of a structured dimension-reduction map called the subsampled randomized Hadamard transform.  This argument demonstrates that the map preserves the Euclidean geometry of an entire subspace of vectors.  The new proof is much simpler than previous approaches, and it offers---for the first time---optimal constants in the estimate on the number of dimensions required for the embedding.
\end{abstract}

\maketitle

\section{Introduction}

\term{Dimension reduction} is an elegant idea from computer science that has found applications in numerical linear algebra.  Here is the basic concept: it is often more efficient to solve a computational problem presented in a high-dimensional space if we first transport the problem instance to a lower-dimensional space while preserving its essential structure.  Researchers have found that randomness provides an extraordinarily effective way to construct these dimension-reduction maps.  This approach is usually traced to the celebrated paper of Johnson and Lindenstrauss~\cite{JL84:Extensions-Lipschitz}.

In randomized algorithms for matrix approximation, the goal of dimension reduction is to find a low-dimensional subspace that captures most of the action of an input matrix.  One way to accomplish this task is to multiply the input matrix $\mtx{A}$ by a relatively small dimension-reduction matrix $\mtx{\Omega}$ to obtain $\mtx{Y} = \mtx{A\Omega}$.  We then perform a QR factorization $\mtx{Y} = \mtx{QR}$ to identify the range of the reduced matrix $\mtx{Y}$.  For an appropriately designed dimension-reduction map, it can be shown that $\mtx{A} \approx \mtx{QQ}^\adj \mtx{A}$ with high probability.
In words, we can approximate the input matrix by compressing it to the range of the reduced matrix.
It is then possible to compute standard matrix decompositions of $\mtx{A}$ by manipulating the low-rank approximation
$\mtx{QQ}^\adj \mtx{A}$.  See the recent survey~\cite{HMT11:Finding-Structure} for a comprehensive treatment of these ideas and an extensive bibliography.

In linear algebra applications, the cost of multiplication by an unstructured random matrix $\mtx{\Omega}$, such as a Gaussian matrix, can sometimes be prohibitive.  In that case, we may prefer to draw the random matrix $\mtx{\Omega}$ from a \emph{highly structured} distribution that allows us to multiply $\mtx{\Omega}$ into the input matrix substantially faster.
Sarl{\'o}s is credited with bringing structured dimension reduction to numerical linear algebra~\cite{Sar06:Improved-Approximation}; see also~\cite{WLRT08:Fast-Randomized}.  

The \term{subsampled randomized Hadamard transform} (SRHT) is a type of
structured dimension-reduction map that is based on the Walsh--Hadamard matrix.
We will prove that the SRHT preserves the geometry of an \emph{entire subspace} of vectors, which is the essential ingredient required to show that the SRHT can be used in algorithms for randomized linear algebra.  See the discussion in~\cite[Sec.~11]{HMT11:Finding-Structure} for details.

The literature already contains a number of papers, including~\cite{AC09:Fast-Johnson-Lindenstrauss,Lib09:Accelerated-Dense,NDT09:Fast-Efficient,HMT11:Finding-Structure,AL11:Almost-Optimal,KW11:New-Improved}, that study the behavior of the SRHT and related dimension reduction maps.  The current treatment differs in several regards.  Here, the main technical difficulties are addressed using a version of the matrix Chernoff inequality~\cite{AW02:Strong-Converse,Tro11:User-Friendly}.
%  In contrast, most previous works rely on a more complicated method due to Rudelson~\cite{Rud99:Random-Vectors} that is based on the noncommutative Khintchine inequality~\cite{L-P86:Inegalites-Khintchine,Buc01:Operator-Khintchine}.
As a consequence of the simple proof schema, we are able---for the first time---to obtain optimal constants in our bounds.  This improvement can be valuable in numerical applications that require concrete performance guarantees.

This document is adapted from Appendix~B of the technical report~\cite{HMT09:Finding-Structure-TR}, but much of the proof is new.  This paper should be viewed as a codicil to the published version~\cite{HMT11:Finding-Structure} of the technical report, which uses the results presented here.

%\subsection{Construction and Properties}

\subsection{Construction of the SRHT Matrix}

%We begin with the basic construction.

An SRHT is a wide $\ell \times n$ matrix of the form 
$$
\mtx{\Phi} = \sqrt{\frac{n}{\ell}} \cdot \mtx{RHD}
$$
where

\begin{itemize}
\item   $\mtx{D}$ is a random $n \times n$ diagonal matrix whose
entries are independent random signs, i.e.,~random variables uniformly distributed on $\{\pm 1\}$;

\item   $\mtx{H}$ is an $n \times n$ Walsh--Hadamard matrix, scaled by $n^{-1/2}$ so it is an orthogonal matrix;

\item   $\mtx{R}$ is a random $\ell \times n$ matrix that
restricts an $n$-dimensional vector to $\ell$ coordinates, chosen
uniformly at random.
\end{itemize}

\noindent
Our analysis relies on two basic properties of the Walsh--Hadamard matrix: The derived matrix $\mtx{H}$ is orthogonal, and its entries all have magnitude $\pm \, n^{-1/2}$.  Walsh--Hadamard matrices exist for each $n = 2^p$ where $p = 1,2,3, \dots$.

%As a consequence of the construction, an SRHT is just a section of an orthogonal matrix, so it satisfies the norm identity $\norm{\mtx{\Phi}} = \sqrt{n/\ell}$.

\begin{rem}
The Walsh--Hadamard matrix is the real analog of the discrete Fourier matrix.
Whereas the $n \times n$ Fourier matrix displays the characters of the cyclic group of order $n$, the $n \times n$ Walsh--Hadamard matrix contains the characters of the additive group $\mathbb{Z}_2^p$ where $n = 2^p$.  In each case, the underlying algebraic structure allows us to multiply the matrix by a vector in about $n \log n$ arithmetic operations.  Note, however, that our results are purely analytic; they do not rely on the algebraic properties of the Walsh--Hadamard matrix.  We have chosen not to work with the Fourier matrix to avoid some small complications associated with complex random variables.
\end{rem}

\subsection{Intuition}

The design of the SRHT may seem mysterious, but there are clear intuitions
to support it~\cite{AC09:Fast-Johnson-Lindenstrauss}.  Suppose that we want to estimate the energy
(i.e., squared $\ell_2$ norm) of a fixed vector $\vct{x}$ by
sampling $\ell$ of its $n$ entries at random.  On average, these
random entries carry $\ell/n$ of the total energy.   (The factor
$\sqrt{n/\ell}$ reverses this scaling.)  When $\vct{x}$
has a few large components, the variance of this estimator
is very high.  On the other hand, when the components of $\vct{x}$
have comparable magnitude, the estimator has much lower variance,
so it is precise with very high probability.

The purpose of the matrix product $\mtx{HD}$ is to flatten out
input vectors before we sample.  To see why it achieves this goal,
fix a unit vector $\vct{x}$, and examine the first component of
$\mtx{HD}\vct{x}$.
$$
(\mtx{HD}\vct{x})_1 = \sum\nolimits_{j=1}^n h_{1j}\eps_j x_j,
$$
where $h_{ij}$ are the components of the matrix $\mtx{H}$.
This random sum clearly has zero mean.  Since the entries
of $\mtx{H}$ have magnitude $n^{-1/2}$, the variance of the sum is $n^{-1}$.  Hoeffding's inequality~\cite{Hoe63:Probability-Inequalities}
shows that
$$
\Prob{ \abs{ (\mtx{HD}\vct{x})_1 } \geq t } \leq 2\econst^{-nt^2/2}.
$$
In words, the magnitude of the first component of $\mtx{HD}\vct{x}$ is typically about $n^{-1/2}$.
The same argument applies to the remaining entries.  Therefore,
it is unlikely that any one of the $n$ components
 of $\mtx{HD}\vct{x}$ is larger than $\sqrt{n^{-1}\log (n)}$.

\begin{rem} \rm
This discussion suggests that the precise form of the SRHT is not
particularly important.  Indeed, we can replace $\mtx{H}$ by any
unitary matrix whose entries are uniformly small and which is
equipped with a fast matrix--vector multiply.  We can also
draw the diagonal entries of $\mtx{D}$ from other subgaussian distributions.
These changes complicate the analysis somewhat.
\end{rem}

\subsection{Main Results}

In early work, dimension reduction was accomplished with uniformly
random projectors or, later,
with Gaussian matrices.
These random maps preserve, up to a constant factor, all the pairwise
distances among $p$ points in $\Rspace{n}$, provided that the
embedding dimension is about $\log(p)$.
We can also use a Gaussian matrix to transport an (unknown) $k$-dimensional
subspace from a high-dimensional space to a lower-dimensional space.
In this case, it suffices to map the ambient space down to $\bigO(k)$ dimensions because
we can discretize a $k$-dimensional sphere using $p = \bigO( \econst^k )$ points.

It turns out that an appropriately designed SRHT also preserves the
geometry of an \emph{entire subspace of vectors}.  For this type of
structured map, the embedding requires $k \log(k)$ dimensions because
of certain phenomena connected with random sampling (Section~\ref{sec:coupons}).
It also becomes necessary to use more sophisticated proof techniques.
We have the following result.

% We present
%a weaker result that is adequate for our purposes.

%An $n \times k$ matrix $\mtx{W}$ is called a \term{partial
%isometry} if its columns are orthonormal.  Evidently, the range of
%$\mtx{W}$ is a $k$-dimensional subspace of $\Cspace{n}$.

\begin{thm}[The SRHT preserves geometry] \label{thm:SRHT-spec-bd}
Fix an $n \times k$ matrix $\mtx{V}$ with orthonormal columns, and
draw a random $\ell \times n$ SRHT matrix $\mtx{\Phi}$
where the embedding dimension $\ell$ satisfies
$$
4 \left[ \sqrt{k} + \sqrt{8\log(kn)} \right]^{2} \log (k) \leq \ell \leq n.
$$
Then, except with probability $\bigO(k^{-1})$,
$$
0.40 \leq \sigma_{k}(\mtx{\Phi V}) 
\quad\text{and}\quad
\sigma_{1}(\mtx{\Phi V}) \leq 1.48.
$$
The symbol $\sigma_j(\cdot)$ denotes the $j$th largest singular
value of a matrix.
\end{thm}

%\notate{Switch to average case analysis???}

%\notate{Rewrite this bit.}

The proof of Theorem~\ref{thm:SRHT-spec-bd} appears below.  This result replaces Theorem~B.4 of~\cite{HMT09:Finding-Structure-TR}.  Some precedents include~\cite[Thm.~3.1]{RV07:Sampling-Large} and \cite[Sec.~9]{Tro08:Conditioning-Random}; see also~\cite{NDT09:Fast-Efficient}.  Earlier
results have the same structure as Theorem~\ref{thm:SRHT-spec-bd},
but the constants are either exorbitant or absent.

\begin{rem}
In Theorem~\ref{thm:SRHT-spec-bd}, the factor $\log(k)$ in the lower bound on $\ell$ cannot generally be removed.  See Section~\ref{sec:coupons} for an explanation and an example that demonstrates the necessity.
\end{rem}

\begin{rem} \rm
For large problems, we have been able to obtain optimal numerical constants.  Suppose that $\iota$ is a sufficiently small positive number.  If $k \gg \log(n)$, then sampling
$$
\ell \geq (1 + \iota) \cdot k \log(k)
$$
coordinates is sufficient to ensure that $\mtx{\Phi V}$ has constant condition number.  See Theorem~\ref{thm:SRHT-large} for a more precise statement.  The discussion in Section~\ref{sec:coupons} indicates that there are cases where it does not suffice to draw $(1 - \iota) \cdot k \log(k)$ samples.
\end{rem}

%\subsection{Performance guarantees}
%
%We are now prepared to present detailed information on the
%performance of the proto-algorithm when the test matrix
%$\mtx{\Omega}$ is an SRFT.
%
%
%\begin{thm}[Error bounds for SRFT]
%\label{thm:SRFT}
%Fix an $m \times n$ matrix $\mtx{A}$ with singular values
%$\sigma_1 \geq \sigma_2 \geq \sigma_3 \geq \dots$.
%Fix a number $k \geq 246$, and draw an $n \times \ell$ SRFT matrix
%$\mtx{\Omega}$, where
%$$
%24 \left[\sqrt{k} + \sqrt{8\log(25n)} \right]^2 \log (4k) \leq \ell \leq n.
%$$
%Construct the sample matrix $\mtx{Y} = \mtx{A\Omega}$. Then
%\begin{align*}
%\norm{ (\Id - \mtx{P}_{\mtx{Y}}) \mtx{A} }
%    &\leq \sqrt{1 + 3n/\ell} \cdot \sigma_{k+1}  \quad\text{and} \\
% \fnorm{ (\Id - \mtx{P}_{\mtx{Y}}) \mtx{A} }
%    &\leq \sqrt{1 + 3n/\ell} \cdot \left( \sum\nolimits_{j > k} \sigma_j^2 \right)^{1/2}
%\end{align*}
%with probability at least $1/2$.
%\end{thm}

\section{Technical Background}

In preparation for the main argument, we present our notation and some probability inequalities.  These inequalities encapsulate all the difficulty in the proof.

\subsection{Notation}

We write $\onevct_j$ for the $j$th standard basis vector in $\mathbb{R}^n$.  The matrix $\Id$ is a square identity matrix; we sometimes indicate the dimension with a subscript.  The symbol ${}^\adj$ denotes transposition.

Throughout this work, $\norm{\cdot}$ refers to the $\ell_2$ vector norm or the associated operator norm.  We also write $\fnorm{\cdot}$ for the Frobenius norm.

Given a subset $T$ of indices in $\{1, 2, \dots,n\}$, we define
the restriction operator $\mtx{R}_{T} : \Rspace{n} \to
\Rspace{T}$ via the rule
$$
(\mtx{R}_{T} \vct{x})(j) = x_{j}, \quad j \in T.
$$

A \term{Rademacher random variable} takes the values $\pm 1$ with
equal probability.  We reserve the letter $\eps$ for a
Rademacher variable, and we often write
$\vct{\eps}$ for a vector whose entries are independent
Rademacher variables.  %We reserve $\eps$ and $\vct{\eps}$ for Rademacher random variables and vectors.

%
%\notate{Missing some notation, e.g., $\Expect^q$ and
%$\Expect_{\vct{\xi}}$.}

%\notate{STILL USING OLDER NOTATION FOR $\ell_2$ NORMS!}

%\subsection{An Alternative Perspective on the SRHT}
%
%Let $n$ be a power of two.  We may construct a random subset $T$ of the set $\{1, 2, \dots, n\}$ by sampling $\ell$ coordinates \emph{without} replacement.
%Given this subset $T$ of indices in $\{1, 2, \dots,n\}$, we define
%the restriction operator $\mtx{R}_{T} : \Cspace{n} \to
%\Cspace{T}$ via the rule
%$$
%(\mtx{R}_{T} \vct{x})(j) = x_{j}, \quad j \in T.
%$$
%In words, $\mtx{R}_{T}$ is a restriction to a random set of coordinates.
%
%For $1 \leq \ell \leq n$, we define the $\ell \times n$ SRHT matrix
%$$
%\mtx{\Phi} = \sqrt{\frac{n}{\ell}} \cdot \mtx{R}_T \mtx{HD},
%$$
%where
%
%\begin{itemize}
%\item   $\mtx{D} = \diag(\vct{\eps})$ is an $n \times n$ diagonal
%matrix with Rademacher entries;
%\item   $\mtx{H}$ is the $n \times n$ Walsh--Hadamard matrix, re-scaled by $1/\sqrt{n}$; and
%\item   $T$ is a random set of $\ell$ coordinates.
%\end{itemize}

\subsection{Probability Inequalities}

We delegate the hard work to some probability inequalities that describe the large-deviation behavior of specific types of random variables.  First, we describe a tail bound for a convex function of Rademacher variates.  This result was established by Ledoux~\cite[Eqn.~(1.9)]{Led96:Talagrands-Deviation}; see also~\cite[\S5.2]{Led01:Concentration-Measure} for a discussion of concentration in product spaces.

%For (smooth) functions of a complex vector, we define convexity
%by means of the gradient inequality with the
%real-linear inner product $\real\ip{\cdot}{\cdot}$.

%We say that a function of a
%vector argument is \term{separately convex} if it is a convex
%function of each coordinate when the rest of the coordinates are
%frozen.

\begin{prop}[Rademacher tail bound] \label{prop:rad-tail}
Suppose $f$ is a convex function on vectors that satisfies
the Lipschitz bound
$$
\abs{ f(\vct{x}) - f(\vct{y}) } \leq L \norm{ \vct{x} - \vct{y} }
\quad\text{for all $\vct{x}, \vct{y}$.}
$$
Let $\vct{\eps}$ be a Rademacher vector.  For all $t \geq 0$,
$$
\Prob{ f(\vct{\eps}) \geq \Expect f(\vct{\eps}) + L t } \leq
\econst^{-t^2/8}.
$$
\end{prop}

Next, we present a matrix analog of the well-known Chernoff inequality.  The proof is based on the matrix Laplace transform method proposed in an influential paper of Ahlswede--Winter~\cite{AW02:Strong-Converse}.  We derive the result using more recent ideas~\cite{Tro11:User-Friendly}, which deliver an essential improvement on the earlier work.  The other key tool is a method~\cite{GN10:Note-Sampling}, ultimately due to Hoeffding~\cite{Hoe63:Probability-Inequalities}, for transferring results from the model where we sample \emph{with} replacement to the model where we sample \emph{without} replacement.

\begin{thm}[Matrix Chernoff] \label{thm:matrix-chernoff}
Let $\coll{X}$ be a finite set of positive-semidefinite matrices with dimension $k$, and suppose that 
$$
\max\nolimits_{\mtx{X} \in \coll{X}} \lambda_{\max}(\mtx{X}) \leq B.
$$
Sample $\{ \mtx{X}_1, \dots, \mtx{X}_\ell\}$ uniformly at random from $\coll{X}$ \emph{without} replacement.  Compute
$$
\mu_{\min} := \ell \cdot \lambda_{\min}\left( \Expect \mtx{X}_1 \right)
\quad\text{and}\quad
\mu_{\max} := \ell \cdot \lambda_{\max}\left( \Expect \mtx{X}_1 \right).
$$
Then
\begin{align*}
\Prob{ \lambda_{\min}\left(\sum\nolimits_j \mtx{X}_j \right)
	\leq (1 - \delta) \mu_{\min} }
	&\leq k \cdot \left[ \frac{\econst^{-\delta}}{(1-\delta)^{1-\delta}} \right]^{\mu_{\min}/B}
\quad\text{for $\delta \in [0, 1)$, and} \\
\Prob{ \lambda_{\max}\left(\sum\nolimits_j \mtx{X}_j \right)
	\geq (1 + \delta) \mu_{\max} }
	&\leq k \cdot \left[ \frac{\econst^{\delta}}{(1+\delta)^{1+\delta}} \right]^{\mu_{\max}/B}
\quad\text{for $\delta \geq 0$.}
\end{align*}
\end{thm}

\begin{proof}[Proof sketch]
We establish only the upper bound; the lower bound is established by applying a similar method to $-\sum\nolimits_j \mtx{X}_j$.  By homogeneity, take $B = 1$.  We use the matrix Laplace transform method~\cite[Prop.~3.1]{Tro11:User-Friendly} to bound the probability of a large deviation:
\begin{equation} \label{eqn:chernoff-prob-bound}
\Prob{ \lambda_{\max}\left(\sum\nolimits_j \mtx{X}_j \right) \geq (1+t)\mu_{\max} }
	\leq \inf_{\theta > 0}\left\{ \econst^{-\theta (1+t) \mu_{\max}} \cdot \Expect \trace \exp\left( \sum\nolimits_j \theta \mtx{X}_j \right) \right\}.
\end{equation}
The Laplace transform bound~\eqref{eqn:chernoff-prob-bound} is essentially due to Ahlswede and Winter~\cite{AW02:Strong-Converse}.

Let $\{\mtx{Y}_1, \dots, \mtx{Y}_\ell\}$ be a random sample from $\coll{X}$ drawn \emph{with} replacement.  Gross and Nesme~\cite{GN10:Note-Sampling} have shown that the trace of the matrix moment generating function (mgf) of a sample \emph{without} replacement is dominated by the trace of the matrix mgf of a sample \emph{with} replacement:
\begin{equation} \label{eqn:without-with}
\Expect \trace \exp\left( \sum\nolimits_j \theta \mtx{X}_j \right)
	\leq \Expect \trace \exp\left( \sum\nolimits_j \theta \mtx{Y}_j \right).
\end{equation}

Lemma~5.8 of~\cite{Tro11:User-Friendly} gives a semidefinite bound for the mgf of $\mtx{Y}_j$.  (See also~\cite[Thm.~19]{AW02:Strong-Converse}.)
\begin{equation} \label{eqn:chernoff-mgf}
\Expect \econst^{\theta \mtx{Y}_j}
	\psdle \econst^{(\econst^{\theta} - 1)( \Expect \mtx{Y}_j )}
	= \econst^{(\econst^{\theta} - 1) (\Expect \mtx{X}_1)}
\quad\text{for $\theta \in \mathbb{R}$.}
\end{equation}
On the right-hand side of this bound, we have exploited the fact that $\mtx{Y}_j \sim \mtx{Y}_1 \sim \mtx{X}_1$.  Lemma~3.4 of~\cite{Tro11:User-Friendly} allows us to use the mgf bound~\eqref{eqn:chernoff-mgf} to bound the mgf of the entire sum:
\begin{equation} \label{eqn:chernoff-final}
\Expect \trace \exp\left( \sum\nolimits_j \theta \mtx{Y}_j \right)
	\leq \trace\exp\left( (\econst^{\theta} - 1) \cdot \ell \cdot ( \Expect \mtx{X}_1) \right)
	\leq k \cdot \exp\left( (\econst^{\theta} - 1) \cdot \mu_{\max} \right).
\end{equation}
Substitute~\eqref{eqn:chernoff-final} into \eqref{eqn:without-with} and introduce the resulting inequality into the probability bound~\eqref{eqn:chernoff-prob-bound}.  The infimum is achieved at $\theta = \log(1 + \delta)$.
\end{proof}

\section{The SRHT Preserves Geometry}

In this section, we establish a slightly more specific version of our main result, Theorem~\ref{thm:SRHT-spec-bd}.  

\begin{thm}[The SRHT preserves geometry] \label{thm:SRHT}
Let $\mtx{V}$ be an $n \times k$ matrix with orthonormal columns.
Select a parameter $\ell$ that satisfies
$$
4 \left[\sqrt{k} + \sqrt{8\log(kn)} \right]^2 \log(k) \leq \ell \leq n.
$$
Draw an $\ell \times n$ SRHT matrix
$\mtx{\Phi}$.  Then, except with probability $3k^{-1}$,
$$
\frac{1}{\sqrt{6}} \leq \sigma_{k}( \mtx{\Phi V} ) 
\quad\text{and}\quad
\sigma_{1}(\mtx{\Phi V} ) \leq \sqrt{\frac{13}{6}}.
$$
\end{thm}

The constants in the sample size $\ell$ specified in Theorem~\ref{thm:SRHT} are somewhat larger than we might like because we wanted to ensure that the statement is effective for reasonable values of $k$ and $n$.  We also have
the following result of a more asymptotic flavor.

\begin{thm}[SRHT: Large sample bounds] \label{thm:SRHT-large}
Fix a positive number $\iota \leq \cnst{c}$.  Let $\mtx{V}$ be an $n \times k$ matrix with orthonormal columns, where $k \geq \cnst{C} \iota^{-2} \log(n)$.  Select a parameter $\ell$ that satisfies
$$
(1 + \iota) \cdot k \log(k) \leq \ell \leq n.
$$
Draw an $\ell \times n$ SRHT $\mtx{\Phi}$.  Then
$$
\iota \leq \sigma_k(\mtx{\Phi} \mtx{V})
\quad\text{and}\quad
\sigma_1(\mtx{\Phi V}) \leq \sqrt{\econst}
$$
except with probability $\bigO(k^{-\cnst{c}\iota})$.
The numbers $\cnst{c}$ and $\cnst{C}$ are positive universal constant.
\end{thm}

\subsection{Overview of Theorem~\ref{thm:SRHT}}

We present the main line of reasoning here, postponing the proofs
of the lemmata.  The first step is to show that the
matrix $\mtx{HD}$ equilibrates row norms.

\begin{lemma}[Row norms] \label{lem:HDV-rows}
Let $\mtx{V}$ be an $n \times k$ matrix with orthonormal columns.  Then
$\mtx{HDV}$ is an $n \times k$ matrix with orthonormal columns, and
$$
\Prob{ \max_{j=1,\dots,n} \norm{ \onevct_j^\adj (\mtx{HDV}) }
    \geq \sqrt{\frac{k}{n}} + \sqrt{\frac{8\log(\beta n)}{n}} }
    \leq \frac{1}{\beta}.
$$
\end{lemma}

When $k \gg \log n$, the second term in the bound is negligible,
in which case the row norms are essentially as small as possible.
On the other hand, when $k < \log n$, small sample effects can
make some row norms large.

The next result states that randomly sampling rows from a matrix with orthonormal columns results in a well-conditioned matrix.  The minimum size of the sample depends primarily on the row norms, and the sampling procedure is
most efficient when the matrix has uniformly small rows.  We state this result in detail because it may have independent interest.

\begin{lemma}[Row sampling] \label{lem:row-samples}
Let $\mtx{W}$ be an $n \times k$ matrix with orthonormal columns, and define the quantity
$M := n \cdot \max_{j=1,\dots,n} \smnorm{}{\onevct_j^\adj \mtx{W}}^2$.
For a positive parameter $\alpha$, select the sample size
$$
\ell \geq \alpha M \log(k).
$$
Draw a random subset $T$ from $\{1,2,\dots, n\}$ by sampling $\ell$ coordinates without replacement.  Then
$$
\sqrt{\frac{(1-\delta) \ell}{n}} \leq \sigma_{k}( \mtx{R}_{T} \mtx{W} )
\quad\text{and}\quad
\sigma_{1}( \mtx{R}_{T} \mtx{W} ) \leq \sqrt{\frac{(1 + \eta) \ell}{n}},
$$
with failure probability at most
$$
k \cdot \left[ \frac{\econst^{-\delta}}{(1-\delta)^{1-\delta}} \right]^{\alpha \log k}
+ k \cdot \left[ \frac{\econst^{\eta}}{(1+\eta)^{(1+\eta)}} \right]^{\alpha \log k}.
$$
\end{lemma}

%We remark that the first term in the failure probability in Lemma~\ref{lem:row-samples} can be simplified by observing that
%$$
%\frac{\econst^{-\delta}}{(1-\delta)^{1-\delta}}
%	\leq \econst^{-\delta^2 / 2}.
%$$
%This bound is obtained by comparison of Taylor series.

%We remark that the failure probability in Lemma~\ref{lem:row-samples} can be bounded above by the more manageable quantity
%\begin{equation} \label{eqn:row-fail}
%k^{1 -\alpha \delta^2 / 2}
%+ k^{1 - \alpha (1+\eta)(\log (1 + \eta) - 1)}.
%\end{equation}
%The first term follows by comparing Taylor series; the second is obtained by replacing $\econst^\eta$ with $\econst^{1+\eta}$ and simplifying.

These two lemmata allow us to establish Theorem~\ref{thm:SRHT} quickly.
Recall that $\mtx{V}$ be an $n \times k$ matrix with orthonormal columns, and
draw an $\ell \times n$ SRHT $\mtx{\Phi} = \mtx{R}_T \mtx{HD}$.

Define the matrix $\mtx{W} = \mtx{HDV}$.  Lemma~\ref{lem:HDV-rows} with $\beta = k$
establishes that $\mtx{W}$ is a matrix with orthonormal columns whose largest row norm
is essentially as small as possible, so that
$$
M := n \cdot \max_{j=1,\dots,n} \smnorm{}{\onevct_j^\adj \mtx{W}}^2
    \leq \left[ \sqrt{k} + \sqrt{8\log(kn)} \right]^2,
$$
except with probability $k^{-1}$.  Next, we apply Lemma~\ref{lem:row-samples}
with $\alpha = 4$, with $\delta = 5/6$, and with $\eta = 7/6$.  A numerical reckoning shows that the additional probability of failure is at most $2k^{-1}$.  Altogether, we obtain the following result.  For
$$
\ell \geq 4 M \log(k),
$$
it holds that
$$
\sqrt{\frac{\ell}{6n}} \leq \sigma_k( \mtx{R}_T \mtx{W} )
\quad\text{and}\quad
\sigma_1( \mtx{R}_T \mtx{W} ) \leq \sqrt{\frac{13\ell}{6n}} 
$$
except with probability $3k^{-1}$.  Since $\mtx{\Phi V} = \sqrt{n/\ell} \cdot \mtx{R}_T \mtx{W}$, we have established Theorem~\ref{thm:SRHT}.

The same type of argument implies Theorem~\ref{thm:SRHT-large}.
In this case, we fix a sufficiently small positive number $\iota$.  Then we take $\beta = k$ and $\alpha = 1 + \iota/2$ and $\delta = 1 - \iota^2$ and $\eta = \econst - 1$.  We omit the details.

\begin{rem}
We can establish a slightly stronger sample bound in Theorem~\ref{thm:SRHT-large} by choosing the parameter $\iota = A / \log(k)$ for a sufficiently large number $A$.  With this selection, however, we do not obtain a constant bound for the lower singular value.
\end{rem}

\subsection{Proofs of Supporting Lemmas}

It remains to check that the underlying results are true.
We begin with the claim that $\mtx{HD}$ balances row norms.

\begin{proof}[Proof of Lemma~\ref{lem:HDV-rows}]
The orthonormality condition on $\mtx{V}$ is equivalent with the identity
$\mtx{V}^\adj \mtx{V} = \Id_k$.  Therefore, $\norm{ \mtx{V} } = 1$
and $\fnorm{\mtx{V}} = \sqrt{k}$.  To check that $\mtx{HDV}$ is
also an orthonormal matrix, simply compute that
$$
(\mtx{HDV})^\adj(\mtx{HDV}) = \mtx{V}^\adj \mtx{V} = \Id
$$
because $\mtx{D}, \mtx{H}$ are orthogonal matrices.

Fix a row index $j \in \{1,2,\dots,n\}$,
and define the function
$$
f(\vct{x})
    := \norm{ \onevct_j^\adj \mtx{H} \diag(\vct{x}) \mtx{V} }
%    = \norm{ \onevct_j^\adj \mtx{H} \diag(\vct{x}) \mtx{V} }
    =: \norm{ \vct{x}^\adj \mtx{EV} }.
$$
We have written $\mtx{E} := \diag( \onevct_j^\adj \mtx{H} )$ for the diagonal
matrix constructed from the $j$th row of $\mtx{H}$; observe that each entry of
$\mtx{E}$ has magnitude $n^{-1/2}$.
The function $f$ is convex, and we quickly
determine its Lipschitz constant:
$$
\abs{ f(\vct{x}) - f(\vct{y}) }
    \leq \norm{ (\vct{x} - \vct{y})^\adj \mtx{EV}  }
    \leq \norm{ \vct{x} - \vct{y} } \norm{ \vct{E} } \norm{ \mtx{V} }
    = \frac{1}{\sqrt{n}} \norm{\vct{x} - \vct{y} }.
$$

We may use the function $f$ to study the variation of the rows norms of $\mtx{HDV}$.  Recall that $\mtx{D} := \diag(\vct{\eps})$ for a Rademacher vector
$\vct{\eps}$, and consider the random variable
$$
f(\vct{\eps}) = \norm{ \onevct_j^\adj \mtx{HDV} }.
$$
First, we bound the expectation.
$$
\Expect f(\vct{\eps})
    \leq \left[ \Expect f(\vct{\eps})^2 \right]^{1/2}
    = \fnorm{ \mtx{EV} }
    \leq \norm{\mtx{E}} \fnorm{\mtx{V}}
    = \sqrt{\frac{k}{n}}.
$$
%since $\mtx{V}$ has Frobenius norm $\sqrt{k}$.
Apply
the Rademacher tail bound, Proposition~\ref{prop:rad-tail}, with $t = \sqrt{8\log(\beta n)}$ to reach
$$
\Prob{ \norm{ \onevct_j^\adj(\mtx{HDV}) } \geq \sqrt{\frac{k}{n}}
+ \sqrt{\frac{8\log(\beta n)}{n}} } \leq \econst^{-8\log( \beta n) / 8} =
\frac{1}{\beta n}.
$$
This estimate holds for each row index $j = 1, 2, \dots, n$.
Finally, take a union bound over these $n$ events to reach the
advertised conclusion.
\end{proof}

Now we establish the result on row sampling.

\begin{proof}[Proof of Lemma~\ref{lem:row-samples}]
Let $\vct{w}_j^\adj$ denote the $j$th row of $\mtx{W}$, and define $M := n \cdot \max_j \normsq{ \vct{w}_j }$.

We can control the extreme singular values of the random matrix
$\mtx{R}_{T} \mtx{W}$ by bounding the extreme eigenvalues of
the $k \times k$ Gram matrix
$$
\mtx{Y} := (\mtx{R}_{T} \mtx{W})^\adj (\mtx{R}_{T} \mtx{W})
	= \sum\nolimits_{j \in T} \vct{w}_j \mtx{w}_j^\adj.
$$
Recall that $T$ is a random subset of $\{ 1, 2, \dots, n\}$ consisting of $\ell$ coordinates sampled without replacement.  Therefore, we may just as well view $\mtx{Y}$ as a sum of $\ell$ random matrices $\mtx{X}_1, \dots, \mtx{X}_\ell$ sampled without replacement from the family $\coll{X} := \{ \vct{w}_j \vct{w}_j^\adj : j = 1, 2, \dots, n \}$ of positive semidefinite matrices with dimension $k$.

The matrix Chernoff bound, Theorem~\ref{thm:matrix-chernoff}, allows us to obtain large deviation bounds for the extreme eigenvalues of $\mtx{Y}$.  We just need to determine the parameters involved in the statement.  First, note that
$$
\lambda_{\max}( \vct{w}_j \vct{w}_j^\adj ) = \normsq{\vct{w}_j} \leq \frac{M}{n}
\quad\text{for $j = 1,2, \dots, n$}.
$$
We easily compute the expectation of the first sample using the fact that the columns of $\mtx{W}$ are orthonormal:
$$
\Expect \mtx{X}_1 = \frac{1}{n} \sum\nolimits_{j=1}^n \vct{w}_j \vct{w}_j^\adj
	= \frac{1}{n} \mtx{W}^\adj\mtx{W}
	= \frac{1}{n} \Id.
$$
As a consequence, we obtain
$$
\mu_{\min} = \frac{\ell}{n}
\quad\text{and}\quad
\mu_{\max} = \frac{\ell}{n}.
$$
The lower Chernoff bound yields
$$
\Prob{ \lambda_{\min}(\mtx{Y})
	\leq (1 - \delta) \frac{\ell}{n} }
	\leq k \cdot \left[ \frac{\econst^{-\delta}}{(1-\delta)^{1-\delta}} \right]^{\ell/M}
\quad\text{for $\delta \in [0, 1)$.}
$$
The upper Chernoff bound yields
$$
\Prob{ \lambda_{\max}(\mtx{Y})
	\geq (1 + \delta) \frac{\ell}{n} }
	\leq k \cdot \left[ \frac{\econst^{\delta}}{(1+\delta)^{1+\delta}} \right]^{\ell/M}
\quad\text{for $\delta \geq 0$.}
$$
Substitute the identities
$$
\lambda_{\min}(\mtx{Y}) = \sigma_{k}(\mtx{R}_T\mtx{W})^2
\quad\text{and}\quad
\lambda_{\max}(\mtx{Y}) = \sigma_{1}(\mtx{R}_T\mtx{W})^2.
$$
Simplify the formulae to complete the proof.
\end{proof}

\subsection{Collecting Coupons} \label{sec:coupons}

The logarithmic factor in these results is \emph{necessary}, as we
show by example.  This discussion is extracted
from~\cite[Sec.~11]{HMT11:Finding-Structure}.

Fix an integer $k$, and set $n = k^2$.  Form an $n \times k$ orthonormal matrix $\mtx{W}$ by regular decimation of the $n \times n$ identity matrix.  More precisely, $\mtx{W}$ is the matrix whose $j$th row has a unit entry in column $1 + (j - 1)/k$ when $j \equiv 1 \pmod{k}$ and is zero otherwise.
To see why this type of matrix is inconvenient, it is helpful to consider an auxiliary matrix $\mtx{V} = \mtx{HD}\mtx{W}$.  Observe that, up to scaling and modulation of rows, $\mtx{V}$ consists of $k$ copies of a $k \times k$ Walsh--Hadamard transform stacked vertically.

Now, suppose that we apply an SRHT $\mtx{\Phi} = \mtx{RHD}$ to the matrix $\mtx{W}$.  We obtain a matrix of the form $\mtx{X} = \mtx{\Phi W} = \mtx{RV}$, which consists of $\ell$ random rows sampled from $\mtx{V}$.
Theorem~\ref{thm:SRHT} cannot hold unless $\sigma_k(\mtx{X}) > 0$.  To ensure the latter event takes place, we must select at least one copy each of the $k$ distinct rows of $\mtx{W}$.  This is the coupon collector's problem~\cite[Sec.~3.6]{MR95:Randomized-Algorithms}.  To obtain a complete set of $k$ rows with nonnegligible probability, we must sample at least $k \log(k)$ rows.  The fact that we are sampling without replacement does not improve the analysis appreciably because the matrix has too many rows.

%\notate{Proof bibliography}

\bibliographystyle{alpha}
\bibliography{srft-final}

\end{document}